\documentclass[preprint,12pt]{elsarticle}
\usepackage{graphicx} 
\usepackage{amsthm}
\usepackage{amsmath}
\usepackage{amssymb}
\usepackage{euscript}
\usepackage{xcolor}
\usepackage{amsfonts}
\usepackage{tikz}
\usepackage{ragged2e}
\usetikzlibrary{arrows}
\usepackage{amsmath}
\usepackage[english]{babel}
\usepackage{graphicx}
\usepackage{mathtools}

\newtheorem{theorem}{Theorem}

\newtheorem{lemma}[theorem]{Lemma}

\newtheorem{dfn}[theorem]{Definition}

\usepackage{theoremref}

\makeatletter
\def\ps@pprintTitle{%
  \let\@oddhead\@empty
  \let\@evenhead\@empty
  \let\@oddfoot\@empty
  \let\@evenfoot\@oddfoot
}
\makeatother
\begin{document}

\begin{frontmatter}

\title{Distributional Chaos in the Baire Space}

\author[J. Mohn]{Jasmin Mohn}
\affiliation[J. Mohn]{organization={Department of Mathematics, United States Military Academy},
            city={West Point},
            state={NY},
            postcode={10996}, 
            country={USA}}
            
\author[B. E. Raines]{Brian E. Raines}
\affiliation[B. E. Raines]{organization={Department of Mathematics, Baylor University},
            city={Waco},
            state={TX},
            postcode={76798--7328}, 
            country={USA}}
\begin{abstract}
In this paper we consider the question of distributional chaos on non-compact metric dynamical systems.  We focus on a shift space over a countable alphabet, the Baire Space. We prove that on the Baire Space subshifts of finite type exhibit dense distributional chaos and subshifts of bounded type that are perfect and have a dense set of periodic points also have distributional chaos.
\end{abstract}


\end{frontmatter}
\section{Introduction}
\label{intro}
The concept of distributional chaos is well-studied. See the next section for relevant definitions in the context of metric spaces. It was first introduced as \emph{strong chaos} in 1994 for continuous mappings on the interval by Schweizer and Sm\'{\i}tal \cite{Schweizer}. In conjunction with Sklar, they showed that a continuous mapping on the interval is distributionally chaotic if and only if it has positive topological entropy \cite{smital}.
However, this does not extend to compact metric spaces as Liao and Fan showed.  There are in fact continuous mappings on a compact metric space with distributional chaos that have zero topological entropy \cite{Liao}. It has been shown that a continuous mapping on a compact metric space has distributional chaos if the space is perfect and has the specification property \cite{SKLAR}.
The extension to compact metric spaces led to the introduction of two generalized forms of distributional chaos, distributional chaos of type 2 and type 3 \cite{three}.  What was before know as distributional chaos is now known as distributional chaos of type 1. These three types of distributional chaos have been found to be equivalent for continuous mappings on trees and graphs \cite{graphs,Risong}. However, this is not the case for dendrites \cite{Roth}.  Stronger versions of distributional chaos such as dense distributional chaos, uniform distributional chaos, and transitive distributional chaos were also introduced by Oprocha \cite{OPROCHA}.
Distributional chaos has since been extended to uniform spaces \cite{Shah}.

In this paper we focus on the question of distributional chaos on non-compact metric spaces. Specifically, we look at a symbolic dynamical system known as the Baire Space, defined in the next section, which is a shift space on a countable alphabet with the product topology and usual shift map. This space is not compact and not locally compact. In fact, it is homeomorphic to the irrationals in $\mathbb{R}$.
We have recently studied $\omega$-chaos in this setting \cite{weak}. In a forthcoming paper, we show that a subshift of the Baire Space with the infinite specification property exhibits distributional chaos of type 1 \cite{infspec}.
The dynamical properties of the shift on this space have been examined by Meddaugh and Raines \cite{BaireSpace}.
A well-studied family of shift spaces over finite alphabets is the family of \emph{subshifts of finite type}, which has been extended to shift spaces over countable alphabet \cite{BaireSpace}.
Another family of shift spaces over countable alphabets is the family of \emph{subshifts of bounded type} introduced by Meddaugh and Raines \cite{BaireSpace}. We will consider both subshifts of finite type and subshifts of bounded type over the Baire Space with the shift map.

The  main results in this paper are that subshifts of finite type exhibit dense distributional chaos, and a subshift of bounded type exhibits distributional chaos on the Baire space if it has a dense set of periodic points and is perfect.

\section{Preliminaries}
\label{prelims}
Let $\omega = \mathbb{N} \cup \{0\}$ be endowed with the discrete topology and
let $X = \omega^\omega$ be endowed with the product topology. Notice that the product topology on $X$ is compatible with the metric $$d(x,y)  = \underset{n \in \omega}{\inf}\{2^{-n}: x_i = y_i \: \forall i<n\}$$
\begin{dfn}Let $\sigma: X \to X$ be the usual shift map, i.e. , $\sigma(x_0x_1x_2\dots)=x_1x_2\dots$. Then $(X, \sigma)$
is called the Baire Space.
\end{dfn}
\begin{dfn}
A set $\Gamma \subseteq \omega^\omega$ is a \textbf{subshift} if it is a closed, invariant subset under the shift map $\sigma$.
\end{dfn}
Let $i,j \in \omega$ with $i< j$.
For $x=x_0x_1x_2\dots \in \Gamma$, define $x_{[i,j)}=x_i x_{i+1} \dots x_{j-1}$.
\begin{dfn}
A word $w$ is said to be an \textbf{allowed word} of length $n$ if and only if there exists $x \in \Gamma$ such that $w = x_{[i, i+n)}$ for some $i \in \omega$.
\end{dfn}
The set of allowed words of length $n$ will be denoted by $B_n(\Gamma)$. Furthermore, $B(\Gamma) = \underset{n\in \omega}{\cup} B_n(\Gamma)$ will denote the set of all allowed words. This allows us to define the following.

\begin{dfn}
For $w \in B(\Gamma)$, the \textbf{cylinder set} of $w$ is $[w] = \{x \in X: x_{[0,|w|)} = w\}$ where $|w|$ is the length of $w$.
\end{dfn}

\begin{dfn}
If $w \notin B(\Gamma)$, we will say $w$ is a \textbf{forbidden word}. We say $\mathcal{F}(\Gamma)$ is a \textbf{basis of forbidden words} if for every forbidden word there exists a subword $w_{[i,j)} = v \in  \mathcal{F}(\Gamma)$ with $0 \leq i < j \leq |w|$.
\end{dfn}
Defining the basis of forbidden words enables us to define the two families of subshifts that are of interest.
\begin{dfn}
A subshift $\Gamma$ is said to be a \textbf{subshift of finite type (SFT)} provided $\Gamma$ has a finite basis of forbidden words.
\end{dfn}
\begin{dfn}
A subshift $\Gamma$ is said to be a \textbf{subshift of bounded type (SBT)} provided there exists $K \in \mathbb{N}$ and a basis of forbidden words of $\Gamma$ such that $|w| < K$ for each $w \in \mathcal{F}(\Gamma)$.
\end{dfn}
The following theorem allows us to discern if certain words are allowed when we have a subshift of bounded type.
\begin{theorem}\thlabel{sbt}
Let $\Gamma$ be a subshift of $\omega^\omega$. Then $\Gamma$ is an SBT if there exists $N \in \omega$ such that if $ab,bc \in B(\Gamma)$ and the length of $b$ is at least $N$, then $abc \in B(\Gamma)$ \cite{BaireSpace}.
\end{theorem}

The main question of this paper is to determine when these families of subshifts of the Baire Space exhibit distributional chaos.
To define distributional chaos of type 1, we first define the function $$\xi(x,y,\epsilon, n) = |\{i : d(\sigma^i(x), \sigma^i(y)) < \epsilon,  i<n\} |$$ where $|\cdot|$ denotes the cardinality of the set \cite{OPROCHA}.
\begin{dfn}
We say $x,y \in X$ form a \textbf{distributionally chaotic pair of type 1 (DC1 pair)} if
$$\limsup_{n \to \infty} \frac{1}{n} \: \xi(x,y,\epsilon, n) =1$$ for all $\epsilon> 0$ and
$$\liminf_{n \to \infty} \frac{1}{n} \: \xi(x,y,\delta, n) =0$$ for some $\delta> 0$.
\end{dfn}

\begin{dfn}
A set $D \subseteq \Gamma$ is said to be \textbf{distributionally scrambled of type 1} if for every pair $x\neq y \in D$, $x,y$ form a distributionally chaotic pair of type 1. If $D$ can be chosen to be uncountable, $\sigma$ is said to be \textbf{distributionally chaotic of type 1 (DC1)}. If $D$ can be chosen to be uncountable and dense in $\Gamma$, $\sigma$ is said to exhibit \textbf{dense distributional chaos}.
\end{dfn}

\section{Results}
\label{results}
First, we consider subshifts of finite type. The following lemma allows us identify a symbol not found in any forbidden word.
\begin{lemma}\thlabel{lemma1}
Let $\Gamma \subseteq \omega^\omega$ be a subshift of finite type. Then there exists $K \in \mathbb{N}$ such that for all $k\geq K$, $k \notin w$ for all $w \in \mathcal{F}(\Gamma)$.
\end{lemma}
\begin{proof}
Let $\Gamma \subseteq \omega^\omega$ be a subshift of finite type.
As $\Gamma$ is a SFT, $\mathcal{F}(\Gamma)$ is finite. Let $$K = 1+ \max \: \{l\in \omega : l = w_i, w=w_0w_1\dots w_{|w|} \in \mathcal{F}(\Gamma)\} $$
Now, if $k \geq K$, then $k > \max \: \{l\in \omega : l = w_i, w \in \mathcal{F}(\Gamma)\}$.
So, $k > l = w_i$ for each $w \in \mathcal{F}(\Gamma)$ and $i \in \{0,1,\dots, |w|\}$.
\end{proof}
In particular, this lemma allows us to identify allowed words. For example, if $a,b \in B(\Gamma)$ and $K$ as in Lemma \ref{lemma1} then $aK^nb \in B(\Gamma)$ is an allowed word for all $n \in \mathbb{N}$. This result aids us in proving Theorem \ref{cyl}.
\begin{theorem}\thlabel{cyl} Let $\Gamma \subseteq \omega^\omega$ be a subshift of finite type. Let $w$ be an allowed word of $\Gamma$. Then the cylinder set, $[w]$, contains an uncountable distributionally scrambled set of type 1.
\end{theorem}
\begin{proof}
Let $w \in B(\Gamma)$ and let $s_0 = 1, m_j = 1+|w|+ \overset{j}{\underset{n=0}{\sum}} s_n,$ and $s_{j+1} = 2^{j+1}m_j $ for all $j \in \omega$.

For all $j,k \in \omega$, define $u(k,j) = k^{s_j}$. Let $K \in \omega$ be as in Lemma \ref{lemma1}. For each $x = x_0x_1x_2 \dots \in \omega^\omega$, define
\begin{align*}
    \hat{x} &:= x_0Kx_0x_1Kx_0x_1x_2K\dots \\
    \quad &\: = \hat{x}_0 \hat{x}_1 \hat{x}_2 \hat{x}_3 \dots
\end{align*}

Define $$D_w = \{\Tilde{x}= w K u(\hat{x}_0,0)u(\hat{x}_1,1)u(\hat{x}_2,2) \dots: x \in \Gamma \text{ and } \hat{x} = \hat{x}_0 \hat{x}_1 \hat{x}_2 \dots \}.$$
Notice that for each $\Tilde{x} \in D_w$, we have the following relationship between $\{s_j\}_{j\in \omega}$, $\{m_j\}_{j\in \omega}$ and the length of the segments of $\Tilde{x}$
$$\Tilde{x}=\underbracket{\underbracket{\underbracket{\underbracket{wK\overbracket{u(\hat{x}_0,0)}^{s_0}}_{m_0}\overbracket{u(\hat{x}_1,1)}^{s_1}}_{m_1} \overbracket{u(\hat{x}_2,2)}^{s_2}}_{m_2}\dots\overbracket{u(\hat{x}_j,j)}^{s_j}}_{m_j}\dots$$

Clearly, $D_w \subseteq [w]$. We show $D_w$ is DC1.

Let $\Tilde{x} \neq \Tilde{y} \in D_w$.
By construction, there exists an increasing sequence $\{\nu_j\}_{j\in\omega}$ such that $\hat{x}_{\nu_j}= K =\hat{y}_{\nu_j}$ for all $j \in \omega$.

Let $q \in \mathbb{N}$ and choose $J$ such that $m_{\nu_j}-m_{\nu_{j-1}} > q$ for all $j\geq J$.
Then for each $m_{\nu_{j-1}}\leq i \leq m_{\nu_j}-q$ $\Tilde{x}_{[i,i+q)}$ is a subword of $u(\hat{x}_{\nu_j},\nu_j)$ of length $q$.
Thus, $\Tilde{x}_{[i,i+q)}= K^q$.
Similarly, $\Tilde{y}_{[i,i+q)}= K^q$.
Hence $$|\{i: \Tilde{x}_{[i,i+q)} = \Tilde{y}_{[i,i+q)} \text{ for } i < m_{\nu_j}\}| \geq |u(\hat{x}_{\nu_j},\nu_j)|-(q+1).$$
Therefore,
\begin{align*}
    \frac{1}{m_{\nu_j}}\xi(\Tilde{x},\Tilde{y}, \frac{1}{2^q}, m_{\nu_j}) &= \frac{1}{m_{\nu_j}} |\{i: \Tilde{x}_{[i,i+q)}=\Tilde{y}_{[i,i+q)} \text{ for } i < m_{\nu_j}\}|\\
    \quad &\geq \frac{|u(\hat{x}_{\nu_j},\nu_j)|-(q+1)}{m_{\nu_j}} \\
    \quad &= \frac{s_{\nu_j}-(q+1)}{m_{\nu_j}} \\
    \quad &=\frac{2^{\nu_j}m_{\nu_j-1}}{m_{\nu_j} } - \frac{ q+1}{m_{\nu_j}}\\
    \quad &=\frac{2^{\nu_j}m_{\nu_j-1}}{1 + |w|+\overset{\nu_j}{\underset{n=0}{\sum}} s_n} - \frac{ q+1}{m_{\nu_j}}\\
    \quad &=\frac{2^{\nu_j}m_{\nu_j-1}}{s_{\nu_j}+( 1 +|w|+\overset{\nu_j-1}{\underset{n=0}{\sum}} s_n )} - \frac{ q+1}{m_{\nu_j}}\\
    \quad &=\frac{2^{\nu_j}m_{\nu_j-1}}{2^{\nu_j}m_{\nu_j-1} +m_{\nu_j-1} } - \frac{q+1}{m_{\nu_j}}\\
    \quad &=\frac{2^{\nu_j}m_{\nu_j-1}}{m_{\nu_j-1}(2^{\nu_j}+1) } - \frac{q+1}{m_{\nu_j}}\\
    \quad &=\frac{2^{\nu_j}}{2^{\nu_j}+1 } - \frac{q+1}{m_{\nu_j}}\\
    & \to 1 \text{ as } j \to \infty
\end{align*}
As $\Tilde{x} \neq \Tilde{y} \in D_w$, $x \neq y$.
Thus $x_l \neq y_l$ for some $l \in \omega$. So, there exists an increasing sequence $\{\mu_j\}_{j\in\omega}$ such that $\hat{x}_{\mu_j}=x_l\neq y_l=\hat{y}_{\mu_j}$ for all $j \in \omega$.
So, $ \Tilde{x}_{i}\neq \Tilde{y}_{i}$ for $m_{\mu_j -1}  \leq i < m_{\mu_j}$.

Thus,
\begin{align*}
    \frac{1}{m_{\mu_j}}\xi(\Tilde{x},\Tilde{y}, 1, m_{\mu_j}) &= \frac{1}{m_{\mu_j}} |\{i: \Tilde{x}_{i}= \Tilde{y}_{i} \text{ for } i < m_{\mu_j}\}|\\
    \quad &\leq \frac{1+|w|+\overset{\mu_j-1}{\underset{n=0}{\sum}} |u(\hat{x}_{n},n)|}{m_{\mu_j}} \\
    \quad &= \frac{\overset{\mu_j-1}{\underset{n=0}{\sum}} s_n}{m_{\mu_j}} +\frac{1+|w|}{m_{\mu_j}}\\
    \quad &=\frac{m_{\mu_j-1}}{m_{\mu_j-1}(1+2^{\mu_j})}+\frac{1+|w|}{m_{\mu_j}}\\
    \quad &=\frac{1}{1+2^{\mu_j}}+\frac{1+|w|}{m_{\mu_j}}\\
    & \to 0 \text{ as } j \to \infty
\end{align*}
Thus, $\Tilde{x},\Tilde{y}$ form a DC1 pair.
Hence, $D_w$ is a distributionally scrambled set of type 1.

We are left to show $D_w$ is uncountable. As $2^\omega$ is uncountable, so is $\{K,K+1\}^\omega \subseteq \Gamma$. Let $\varphi: \{K,K+1\}^\omega \to D_w$ be defined by $\varphi(x) = w K u(\hat{x}_0,0)u(\hat{x}_1,1)u(\hat{x}_2,2) \dots$.
Suppose $\varphi(x) = \varphi(y)$. Then, $u(\hat{x}_j,j) = u(\hat{y}_j,j)$ for all $j \in \omega$. Hence $x_j = y_j$ for all $j \in \omega$ and $x = y$ as desired.
\end{proof}
Thus, the shift map on a subshift of finite type exhibits distributional chaos. To show that a subshift of finite type exhibits dense distributional chaos, we need the following lemmas.
\begin{lemma}\thlabel{lemma2} Let $K$ be as in Lemma \ref{lemma1}. Choose $g,h,p \in \mathbb{N}$.
Let $x \in \{2(K+g),2(K+g)+p\}^\omega$ and $y \in \{2(K+h),2(K+h)+p\}^\omega$ where $x,y$ are not fixed points. That is, $x\not\in \{ (2(K+g))^\infty, (2(K+g)+p)^\infty\}$ and $y\not\in \{(2(K+h))^\infty, (2(K+h)+p)^\infty\}$. If $g \neq h$, then $x \neq y$. That is, there exists $\gamma \in \omega$ such that $x_\gamma \neq y_\gamma$.
\end{lemma}
\begin{proof} Let $g \neq h \in \mathbb{N}.$
Suppose $x \in \{2(K+g),2(K+g)+p\}^\omega$ and $y \in \{2(K+h),2(K+h)+p\}^\omega$ with $x\not\in \{(2(K+g))^\infty, (2(K+g)+p)^\infty\}$ and $y\not\in \{(2(K+h))^\infty, (2(K+h)+p)^\infty\}$.

Suppose toward contradiction $x = y$. As $x\neq (2(K+g))^\infty$ and $x\neq (2(K+g)+p)^\infty$, there exists $m,n \in \omega$ such that $x_m = 2(K+g)$ and $x_n = 2(K+g)+p$. By assumption $g \neq h$, so $y_m = 2(K+h)+p$. This gives us $2(K+g) = x_m = y_m = 2(K+h)+p $ and hence $x_n = 2(K+g)+p = 2(K+h)+2p$. As $g \neq h$, we also have $y_n = 2(K+h)$. Thus, $2(K+h) + 2p = x_n = y_n = 2(K+h)$. Therefore, $p=0$. This is a contradiction.
\end{proof}

\begin{lemma}\thlabel{lemma3}
Let $K$ be as in Lemma \ref{lemma1}. Choose $g,h,p,q \in \mathbb{N}$.
Let $x \in \{2(K+g),2(K+g)+p\}^\omega$ and $y \in \{2(K+h),2(K+h)+q\}^\omega$ where $x,y$ are not fixed points. That is, $x\not\in \{(2(K+g))^\infty, (2(K+g)+p)^\infty\}$ and $y\not\in \{(2(K+h))^\infty, (2(K+h)+q)^\infty\}$. If $p \neq q$, then $x \neq y$. That is, there exists $\gamma \in \omega$ such that $x_\gamma \neq y_\gamma$.
\end{lemma}
\begin{proof}
Let $p \neq q \in \mathbb{N}.$
Suppose $x \in \{2(K+g),2(K+g)+p\}^\omega$ and $y \in \{2(K+h),2(K+h)+q\}^\omega$ with $x\not\in \{(2(K+g))^\infty, (2(K+g)+p)^\infty\}$ and $y\not\in \{(2(K+h))^\infty, (2(K+h)+q)^\infty\}$.
Suppose toward contradiction $x=y$. As $x\neq (2(K+g))^\infty$ and $x\neq (2(K+g)+p)^\infty$, there exists $m,n \in \omega$ such that $x_m = 2(K+g)$ and $x_n = 2(K+g)+p$.

Case 1: $y_m =2(K+h)$. Then $g=h$ as $x_m = y_m$. So, $y_n = x_n = 2(K+g)+p = 2(K+h)+p$. As $p\neq q$, we have $y_n = 2(K+h)$. Thus, $p=0$ This is a contradiction.

Case 2: $y_m = 2(K+h)+q$. Then $y_n = x_n = 2(K+g)+p = x_m + p =2(K+h)+q+p$. As $x_m \neq x_n$, we also have $y_n = 2(K+h)$. Thus, $p+q = 0$. Hence, $p,q =0$. This is a contradiction.
\end{proof}

 Now, we are ready to show that the shift map on a subshift of finite type exhibits dense distributional chaos.

\begin{theorem}
Let $\Gamma \subseteq \omega^\omega$ be a subshift of finite type. $\sigma$ exhibits dense distributional chaos.
\end{theorem}

\begin{proof}
Let $s_0 = 1, m_j = \overset{j}{\underset{n=0}{\sum}} s_n,$ and $s_{j+1} = 2^{j+1}m_j$ for all $j \in \omega$.

For all $k \in \omega$, define $u(k,j) = k^{s_j}$.
By Lemma \ref{lemma1}, there exists $K \in \omega$ such that for all $k \geq K$, $k \notin w$ if $w \in \mathcal{F}(\Gamma)$. For each $x = x_0x_1x_2 \dots \in \omega^\omega$, define $$\hat{x} = x_0Kx_0x_1Kx_0x_1x_2K\dots$$
For each $p \in \omega$, $\omega^p$ is countable. Thus, $B_p(\Gamma)$ is countable. Hence, $B_p(\Gamma)$ can be enumerated by $B_p(\Gamma)= \{w^{p,g}\}_{g \in \omega}$.
For each $g \in \omega$, we define $r_g = 2(K+g)$ and 
$$D_{w^{p,g}} = \{\Tilde{x}=w^{p,g}Ku(\hat{x}_0,0)u(\hat{x}_1,1)\dots : x \in \{r_g, r_g+p\}^\omega \setminus \{(r_g)^\infty, (r_g+p)^\infty\} \}$$
Notice that for each $\Tilde{x} \in D_{w^{p,g}}$, we have the following relationship between $\{s_j\}_{j\in \omega}$, $\{m_j\}_{j\in \omega}$ and the length of the segments of $\Tilde{x}$
$$\Tilde{x}=\underbracket{\underbracket{\underbracket{\underbracket{\underbracket{w^{p,g}K}_{p+1}\overbracket{u(\hat{x}_0,0)}^{s_0}}_{p+1+m_0}\overbracket{u(\hat{x}_1,1)}^{s_1}}_{p+1+m_1} \overbracket{u(\hat{x}_2,2)}^{s_2}}_{p+1+m_2}\dots\overbracket{u(\hat{x}_j,j)}^{s_j}}_{p+1+m_j}\dots$$

We now show $D = \underset{p \in \omega}{\cup}\underset{g \in \omega}{\cup} D_{w^{p,g}}$ is a dense DC1 set.
Let $\Tilde{x} \in D_{w^{p,g}}$ and $\Tilde{y} \in D_{w^{q,h}}$ such that $\Tilde{x} \neq \Tilde{y}$. We will show $\Tilde{x}, \Tilde{y}$ form a DC1 pair.

Case 1: $p = q$. That is, $|w^{p,g}| = |w^{q,h}| = p$. By construction, there exists an increasing sequence $\{\nu_j\}_{j\in\omega}$ such that $\hat{x}_{\nu_j}= K =\hat{y}_{\nu_j}$ for all $j \in \omega$.

\begin{itemize}
\item[]
\begin{itemize}
        \item[Subcase 1:] $w^{p,g}=w^{p,h}$. Then, $i = j$. As $\Tilde{x} \neq \Tilde{y}$, $x \neq y$. That is, $x_\gamma \neq y_\gamma$ for some $\gamma \in \omega$. Hence, there exists an increasing sequence $\{\mu_j\}_{j\in\omega}$ such that $\hat{x}_{\mu_j}\neq\hat{y}_{\mu_j}$ for all $j \in \omega$.
        \item[Subcase 2:] $w^{p,g} \neq w^{p,h}$. Then, $i \neq j$. So, $x \in (2(K+g),2(K+g)+p)^\omega$ and $y \in (2(K+h),2(K+h)+p)^\omega$ where $x,y$ are not fixed points.
By Lemma \ref{lemma2}, there exists $\gamma \in \omega$ such that $x_\gamma \neq y_\gamma$. Hence, there exists an increasing sequence $\{\mu_j\}_{j\in\omega}$ such that $\hat{x}_{\mu_j}\neq\hat{y}_{\mu_j}$ for all $j \in \omega$.
\end{itemize}
\end{itemize}

Now, let $t \in \mathbb{N}$ and choose $J$ large enough so that $m_{\nu_j} - m_{\nu_{j-1}}>t$ for all $j \geq J$.
Notice $$\frac{1}{p+1+m_{\nu_j}}\xi(\Tilde{x},\Tilde{y}, \frac{1}{2^t}, p+1+m_{\nu_j}) $$ equals $$ \frac{1}{p+1+m_{\nu_j}} |\{i: \Tilde{x}_{[i,i+t)}=\Tilde{y}_{[i,i+t)} \text{ for } i < p+1+m_{\nu_j}\}|.$$
Thus,
\begin{align*}
    \frac{1}{p+1+m_{\nu_j}}\xi(\Tilde{x},\Tilde{y}, \frac{1}{2^t}, p+1+m_{\nu_j}) &\geq \frac{1}{p+1+m_{\nu_j}}\big(|u(\hat{x}_{\nu_j},\nu_j)|-(t+1)\big) \\
    \quad &= \frac{s_{\nu_j}-(t+1)}{p+1+m_{\nu_j}} \\
    \quad &=\frac{2^{\nu_j}m_{\nu_j-1}}{m_{\nu_j-1}(1 + 2^{\nu_j}+ \frac{p+1}{m_{\nu_j-1}}) } - \frac{t+1}{p+1+m_{\nu_j}}\\
    \quad &=\frac{2^{\nu_j}}{1 + 2^{\nu_j}+ \frac{p+1}{m_{\nu_j-1}}} - \frac{t+1}{p+1+m_{\nu_j}}\\
    & \to 1 \text{ as } j \to \infty
\end{align*}
As $$\frac{1}{p+1+m_{\mu_j}}\xi(\Tilde{x},\Tilde{y}, 1, p+1+m_{\mu_j})$$ equals $$ \frac{1}{p+1+m_{\mu_j}} |\{i: \Tilde{x}_{i}= \Tilde{y}_{i} \text{ for } i <p+1+ m_{\mu_j}\}|,$$
we have 
\begin{align*}
    \frac{1}{p+1+m_{\mu_j}}\xi(\Tilde{x},\Tilde{y}, 1, p+1+m_{\mu_j}) &\leq \frac{p+1+\overset{\mu_j-1}{\underset{n=0}{\sum}} |u(\hat{x}_{n},n)|}{p+1+m_{\mu_j}} \\
    \quad &= \frac{\overset{\mu_j-1}{\underset{n=0}{\sum}} s_n}{p+1+m_{\mu_j}} + \frac{p+1}{p+1+m_{\mu_j}} \\
    \quad &=\frac{m_{\mu_j-1}}{m_{\mu_j-1}(1+2^{\mu_j}+\frac{p+1}{m_{\mu_j-1}})}+ \frac{p+1}{p+1+m_{\mu_j}} \\
    \quad &=\frac{1}{1+2^{\mu_j}+\frac{p+1}{m_{\mu_j-1}}}+ \frac{p+1}{p+1+m_{\mu_j}} \\
    & \to 0 \text{ as } j \to \infty
\end{align*}
Thus in the case that $p=q$, $\Tilde{x},\Tilde{y}$ form a DC1 pair.

Case 2: $p \neq q$. That is, $|w^{p,g}| \neq |w^{q,h}|$. Without loss of generality, assume $p > q$. Then, $p= q + a$ for some $a$. Notice there exists $B$ such that for all $b \geq B$, $s_b > a$.
Notice for all $b \geq B$ we have $$\Tilde{x}_{[p + 1 + m_{b-1}, p + 1 + m_{b})} = u(\hat{x}_b, b)=\hat{x}_b^{s_b}$$ and $$\Tilde{y}_{[q + 1 + m_{b-1}, q + 1 + m_{b})} = u(\hat{y}_b, b)= \hat{y}_b^{s_b}$$ Thus, for all $b \geq B$ $$\Tilde{x}_{[p + 1 + m_{b-1}, p + 1 + m_{b}-a)} = \hat{x}_b^{s_b - a}$$
As $p > q$, we have $$p+1+m_{b-1} > q+1 + m_{b-1}$$
The choice of $b > B$ guarantees $$p+1+m_{b}-a > p+1+m_{b} - s_b = p+1+m_{b-1}$$ as $s_b > a$.
Lastly, we have $$p+1+m_{b} - a = q+a+1+m_{b}-a =q+1+m_{b}$$ as $p=q+a$.

Thus, for all $b \geq B$ $$\Tilde{y}_{[p + 1 + m_{b-1}, p + 1 + m_{b}-a)} = \hat{y}_b^{s_b - a}$$
So, by construction, there exists an increasing sequence $\{\nu_j\}_{j\in\omega}$ such that $\hat{x}_{\nu_j}= K =\hat{y}_{\nu_j}$ for all $j \in \omega$.
As $p \neq q$, by Lemma \ref{lemma3} $x \neq y$. That is, $x_\gamma \neq y_\gamma$ for some $\gamma\in \omega$. Thus, there exists an increasing sequence $\{\mu_j\}_{j\in\omega}$ such that $\hat{x}_{\mu_j}= x_\gamma \neq y_\gamma = \hat{y}_{\mu_j}$ for all $j \in \omega$.
Now, let $t \in \mathbb{N}$ and choose $J$ large enough so that $m_{\nu_j} - m_{\nu_{j-1}}>t$ for all $j\geq J$.
As $$ \frac{1}{p+1+m_{\nu_j}}\xi(\Tilde{x},\Tilde{y}, \frac{1}{2^t},p+1+ m_{\nu_j}) $$ equals $$ \frac{1}{p+1+m_{\nu_j}} |\{i: \Tilde{x}_{[i,i+t)}=\Tilde{y}_{[i,i+t)} \text{ for } i <p+1+ m_{\nu_j}\}|$$,
\begin{align*}
    \frac{1}{p+1+m_{\nu_j}}\xi(\Tilde{x},\Tilde{y}, \frac{1}{2^t},p+1+ m_{\nu_j}) &\geq \frac{1}{p+1+m_{\nu_j}}\big(|u(\hat{x}_{\nu_j},\nu_j)|-(t+1+a)\big) \\
    \quad &= \frac{s_{\nu_j}-(t+1+a)}{p+1+m_{\nu_j}} \\
    \quad &=\frac{2^{\nu_j}m_{\nu_j-1}}{m_{\nu_j-1}(1 + 2^{\nu_j}+\frac{p+1}{m_{\nu_j-1}}) } - \frac{t+1+a}{p+1+m_{\nu_j}}\\
    \quad &=\frac{2^{\nu_j}}{1 + 2^{\nu_j}+\frac{p+1}{m_{\nu_j-1}} } - \frac{t+1+a}{p+1+m_{\nu_j}}\\
    & \to 1 \text{ as } j \to \infty
\end{align*}
Now,
\begin{align*}
    \frac{1}{p+1+m_{\mu_j}}\xi(\Tilde{x},\Tilde{y}, 1, p+1+m_{\mu_j}) &= \frac{1}{m_{\mu_j}} |\{i: \Tilde{x}_{i} = \Tilde{y}_{i} \text{ for } i < p+1+m_{\mu_j}\}|\\
    \quad &\leq \frac{a+p+1+\overset{\mu_j-1}{\underset{n=0}{\sum}} |u(\hat{x}_{n},n)|}{p+1+m_{\mu_j}} \\
    \quad &= \frac{\overset{\mu_j-1}{\underset{n=0}{\sum}} s_n}{p+1+m_{\mu_j}} + \frac{a+p+1}{p+1+m_{\mu_j}} \\
    \quad &=\frac{m_{\mu_j-1}}{m_{\mu_j-1}(1+2^{\mu_j}+\frac{p+1}{m_{\mu_j-1}})}+ \frac{a+p+1}{p+1+m_{\mu_j}} \\
    \quad &=\frac{1}{1+2^{\mu_j}+\frac{p+1}{m_{\mu_j-1}}}+ \frac{a+p+1}{p+1+m_{\mu_j}} \\
    & \to 0 \text{ as } j \to \infty
\end{align*}

Thus in the case that $p \neq q$, $\Tilde{x},\Tilde{y}$ also form a DC1 pair.

To show $D$ is dense in $\Gamma$, let $z \in \Gamma$. For each $p \in \omega$, $z_{[0,p)}= w^{p,g}$ for some $w^{p,g} \in B_p(\Gamma)$. For each $p \in \omega$, choose any $\Tilde{x}^p \in D_{w^{p,g},r_g}$.
This defines a sequence $\{\Tilde{x}^p\}_{p \in \omega}$ in $D$ such that $\Tilde{x}^p \to z$. Therefore, $D$ is dense in $\Gamma$.

We are left to show $D$ is uncountable. As $2^\omega$ is uncountable, so is $\{r_g,r_g+p\}^\omega$. Hence $\{r_g,r_g+p\}^\omega\setminus \{(r_g)^\infty, (r_g+p)^\infty\}$ is also uncountable. Define $\varphi : \{r_g,r_g+p\}^\omega\setminus \{(r_g)^\infty, (r_g+p)^\infty\} \to D_{w^{p,g}}$ by $\varphi(x) = w_{p,g}Ku(\hat{x}_0,0)u(\hat{x}_1,1)\dots$. We will show $\varphi$ is a bijective. If $\varphi(x) = \varphi(y)$, $\Tilde{x}= \Tilde{y}$. So, $u(\hat{x}_j,j) = u(\hat{y}_j,j)$ for all $j \in \omega$. Hence $x_j = y_j$ for all $j \in \omega$ and $x = y$ as desired. Note $\varphi$ is surjective by construction. Therefore, $D_{w^{p,g}}$ is uncountable. As $D = \cup_{p \in \omega}\cup_{g \in \omega} D_{w^{p,g}} \supseteq D_{w^{p,g}}$, $D$ is uncountable as desired. Thus, $\sigma$ exhibits dense distributional chaos.
\end{proof}

Now, let us consider a subshift of bounded type.

\begin{theorem}
Let $\Gamma \subseteq \omega^\omega$ be a subshift of bounded type. If $\Gamma$ has a dense set of periodic points and is perfect, then $\Gamma$ contains an uncountable scrambled set of type 1.
\end{theorem}
\begin{proof}
As $\Gamma$ is a subshift of bounded type, there exists $L \in \mathbb{N}$ and a basis of forbidden words of $\Gamma$, $\mathcal{F}(\Gamma)$, such that $|w| < L$ for each $w \in \mathcal{F}(\Gamma)$. Let $N$ be as in Theorem \ref{sbt} and let $K = \max\{L,N\}$. Fix $z \in \Gamma$ periodic with prime period $p> K!$. Note $z$ exists as periodic points are dense in $\Gamma$ and $\Gamma$ is uncountable.

Let $\epsilon > 0$ be given so that if $d(x,y) < \epsilon$ then $x_{[0,p)}=y_{[0,p)}$.
Choose $x \in \Gamma$ with prime period $q > p$ such that $d(x,z) < \epsilon$ and choose $y \in \Gamma$ with prime period $r > q$ such that $d(y,z) < \epsilon$.
Notice that $$x = z_{[0,p)}x_{[p,pq)}z_{[0,p)}x_{[p,pq)} \dots$$ and
$$y = z_{[0,p)}y_{[p,pr)}z_{[0,p)}y_{[p,pr)} \dots$$
Define $a=z_{[0,p)},b=x_{[p,pq)},$ and $c=y_{[p,pr)}$.
Notice $|b| = pq - p = p(q-1) \geq p^2 > K$ and $|c| = pr - p = p(r-1) > p^2 > K$.
As $ab, ba, ac,ca \in B(\Gamma)$ and $|a|, |b|, |c| > K = \max\{L,N\} \geq N$, by  Theorem \ref{sbt} $aba, cab, cac \in B(\Gamma)$.

As $r > q > p$ are prime periods for $y,x,z$ respectively, there exist $\gamma_1,\gamma_2, \gamma_3 \in \omega$ such that $\gamma_1, \gamma_2 < r$, $\gamma_3< q$, $x_{\gamma_1} \neq y_{\gamma_1}$,  $y_{\gamma_2} \neq z_{\gamma_2}$,  and $x_{\gamma_3} \neq z_{\gamma_3}$.

Now, choose $A,B,C \in \mathbb{N}$ minimal such that $$A|a|=B|ba|=C|ca|$$
and let $M = A|a|$.
Define $ I_{\alpha_i}=
\begin{cases}
(ba)^B & \text{if }\alpha_i=0 \\
(ca)^C & \text{if }\alpha_i=1 \\
\end{cases}$.

Now, let $s_0 = 1, m_j = \overset{j}{\underset{n=0}{\sum}} M s_n,$ and $s_{j+1} = 2^{j+1}Mm_j$ for all $j\in \omega$.
For all $t \in B(\Gamma)$ and $n\in\omega$, define $u(t,n) = (t)^{s_n}$ and $v(t,n) = (t)^{As_n}$.

For each $\alpha \in 2^\omega$, define $$\varphi (\alpha)= u(I_{\alpha_0},0) v(a,1) u(I_{\alpha_0},2) u(I_{\alpha_1},3)v(a,4) \dots$$ where $a = z_{[0,p)}$ as before
and let $D=\{\varphi(\alpha) : \alpha \in 2^\omega\}$.
Notice that for each $\varphi(\alpha) \in D$, we have the following relationship between $\{s_j\}_{j\in \omega}$, $\{m_j\}_{j\in \omega}$ and the length of the segments of $\varphi(\alpha)$
$$\varphi(\alpha)=\underbracket{\underbracket{\underbracket{\underbracket{\overbracket{u(I_{\alpha_0},0)}^{Ms_0}}_{m_0}\overbracket{v(a,1)}^{Ms_1}}_{m_1} \overbracket{u(I_{\alpha_0},2)}^{Ms_2}}_{m_2}\overbracket{u(I_{\alpha_1},3)}^{Ms_3}}_{m_3}\dots$$

Suppose $\alpha \neq \beta \in 2^\omega$. Then $\alpha_i \neq \beta_i$ for some $i \in \mathbb{N}$. So $I_{\alpha_i}\neq I_{\beta_i}$.
Construct an increasing sequence $\{\mu_j\}_{j \in \omega} \subseteq \mathbb{N}$ consisting of every $n \in \mathbb{N}$ such that $u(I_{\alpha_i},\mu_j) = u(I_{\alpha_i},n)$ occurs in $\varphi(\alpha)$.

Let $\Tilde{x}=\varphi(\alpha)$ and $\Tilde{y}=\varphi(\beta)$.

Choose $\mu_j$ from the sequence defined above.
Notice for $m_{\mu_{j-1}} \leq l < m_{\mu_j}-M$, $$\Tilde{x}_{[l, l+M)}=\big(u(I_{\alpha_i},\mu_j)\big)_{[l-m_{\mu_{j-1}}, l-m_{\mu_{j-1}}+M)}$$ and $$\Tilde{y}_{[l, l+M)}=\big(u(I_{\beta_i},\mu_j)\big)_{[l-m_{\mu_{j-1}}, l-m_{\mu_{j-1}}+M)}.$$ As $ca$ is not an initial segment of $(ba)^\infty$, there exists $\theta \in \{0, 1, \dots, M-1\}$ such that $(ba)^B_\theta \neq (ca)_\theta^C$.
\\Case 1: Suppose
$$\Tilde{x}_{[l, l+M)}= \big((ba)^{Bs_{\mu_j}}\big)_{{[l-m_{\mu_{j-1}}, l-m_{\mu_{j-1}}+M)}} = (ba)^B$$
Then,
$$ \Tilde{y}_{[l, l+M)} = \big((ca)^{Cs_{\mu_j}}\big)_{{[l-m_{\mu_{j-1}}, l-m_{\mu_{j-1}}+M)}} =(ca)^C $$
Since $(ba)^B_\theta \neq (ca)_\theta^C$, $\Tilde{x}_{[l, l+M)}\neq \Tilde{y}_{[l, l+M)}$.
\\Case 2: There exists $e \in \mathbb{N}$ such that
$$\Tilde{x}_{[l, l+M)} = \big((ba)^{B\mu_j}\big)_{[l-m_{\mu_{j-1}}, l-m_{\mu_{j-1}}+M)} = \big((ba)^B)_{[e,M)}\big((ba)^B)_{[0,e)}$$
So,
$$ \Tilde{y}_{[l, l+M)}= \big((ca)^{C\mu_j}\big)_{[l-m_{\mu_{j-1}}, l-m_{\mu_{j-1}}+M)} =\big((ca)^C)_{[e,M)}\big((ca)^C)_{[0,e)}$$
Since $\theta \in \{0,1,\dots,j-1\}$ or $\theta \in \{j,j+1,\dots,M-1\}$, $\Tilde{x}_{[l, l+M)}\neq \Tilde{y}_{[l, l+M)}$.

Therefore in either case for all $m_{\mu_{j-1}} \leq l < m_{\mu_j}-M$,
 $\Tilde{x}_{[l, l+M)}\neq \Tilde{y}_{[l, l+M)}$.
Hence,
\begin{align*}
    \frac{1}{m_{\mu_j}}\xi(\Tilde{x},\Tilde{y}, \frac{1}{2^{M}}, m_{\mu_j}) &= \frac{1}{m_{\mu_j}} |\{i: \Tilde{x}_{[i,i+M)}= \Tilde{y}_{[i,i+M)} \text{ for } i < m_{\mu_j}\}|\\
    \quad &\leq \frac{\overset{\mu_j-1}{\underset{n=0}{\sum}} Ms_n + M}{m_{\mu_j}} \\
    \quad &= \frac{m_{\mu_j-1}}{m_{\mu_j}} + \frac{M}{m_{\mu_j}} \\
    \quad &=\frac{m_{\mu_j-1}}{m_{\mu_j-1}(1+2^{\mu_j})} + \frac{M}{m_{\mu_j}}\\
    \quad &=\frac{1}{1+2^{\mu_j}} + \frac{M}{m_{\mu_j}}\\
    & \to 0 \text{ as } j \to \infty
\end{align*}

We can also construct an increasing sequence $\{\nu_j\}_{j\in\omega} \subseteq \mathbb{N}$ consisting of every $n \in \mathbb{N}$ such that $v(a,\nu_j) = v(a,n)$ occurs in $\varphi(\alpha)$. Notice at these places $v(a,n)$ also occurs in $\varphi(\beta)$.
Let $R \in \mathbb{N}$ be given. Then,
\begin{align*}
    \frac{1}{m_{\nu_j}}\xi(\Tilde{x},\Tilde{y}, \frac{1}{2^{R}}, m_{\nu_j}) &= \frac{1}{m_{\nu_j}} |\{i: \Tilde{x}_{[i,i+R)}= \Tilde{y}_{[i,i+R)} \text{ for } i < m_{\nu_j}\}|\\
    \quad &\geq \frac{Ms_{\nu_j} - (R+1)}{m_{\nu_j}} \\
    \quad &= \frac{2^{\nu_j}M^2 m_{\nu_j - 1}}{m_{\nu_j-1}(1 + 2^{\nu_j}M^2)} - \frac{R+1}{m_{\nu_j}} \\
    \quad &= \frac{2^{\nu_j}M^2}{1 + 2^{\nu_j}M^2} - \frac{R+1}{m_{\nu_j}} \\
    & \to 1 \text{ as } j \to \infty
\end{align*}
Thus, $\Tilde{x},\Tilde{y}$ form a DC1 pair.
Lastly, as $\varphi$ is injective $D$ is uncountable.
\end{proof}

\bibliographystyle{Distributional_Chaos_in_the_Baire_Space} 
\bibliography{Distributional_Chaos_in_the_Baire_Space}
\end{document}